\documentclass[11pt]{article}

\usepackage{amssymb, amsmath, fullpage, amsthm}
\usepackage{mathrsfs}

\newtheorem{theorem}{Theorem}[section]
\newtheorem{lemma}[theorem]{Lemma}
\newtheorem{proposition}[theorem]{Proposition}
\newtheorem{definition}[theorem]{Definition}
\newtheorem{corollary}[theorem]{Corollary}
\newtheorem{remark}[theorem]{Remark}

\numberwithin{equation}{section}

\newcommand{\hz}{\widehat{0}}

\newcommand{\timesdots}{\times \cdots \times}

\DeclareMathOperator{\inv}{inv}
\DeclareMathOperator{\frst}{frst}
\DeclareMathOperator{\er}{er}
\DeclareMathOperator{\asc}{asc}
\DeclareMathOperator{\odd}{odd}

\newcommand{\ascodd}{\asc_{\odd}}

\newcommand{\doubleprime}{\prime\prime}

\newcommand{\SSSS}{\mathfrak{S}}

\newcommand{\Gaussian}[2]{\genfrac{[}{]}{0pt}{}{#1}{#2}_q}

\begin{document}

\title{The Gaussian coefficient revisited}

\author{\sc Richard EHRENBORG\thanks{Corresponding author:
Department of Mathematics,
University of Kentucky,
Lexington, KY 40506-0027,
USA,
{\tt richard.ehrenborg@uky.edu}.}
\:\: and 
\:\:
Margaret A.\ READDY\thanks{Department of Mathematics,
University of Kentucky,
Lexington, KY 40506-0027,
USA,
{\tt margaret.readdy@uky.edu}.}
}

\date{}

\maketitle

\begin{abstract}
We give a new $q$-$(1+q)$-analogue of the
Gaussian coefficient, also known as the
$q$-binomial which, like the original $q$-binomial
$\Gaussian{n}{k}$,
is symmetric in $k$ and $n-k$.
We show this $q$-$(1+q)$-binomial is more compact than
the one discovered by Fu, Reiner, Stanton and Thiem.
Underlying our $q$-$(1+q)$-analogue
is a Boolean algebra decomposition
of an associated poset.
These ideas are extended to the Birkhoff transform of any finite poset.
We end with a discussion of higher analogues of the
$q$-binomial.

\vspace*{2 mm}


\vspace*{2 mm}

\end{abstract}

\section{Introduction}

Inspired by work of Fu, Reiner, Stanton 
and Thiem~\cite{Fu_Reiner_Stanton_Thiem},
Cai and Readdy~\cite{Cai_Readdy}
asked the following question.
Given a combinatorial $q$-analogue
$$     X(q)   =   \sum_{w \in X} q^{a(w)}   ,   $$
where $X$ is a set of objects and
$a(\cdot)$ is a  statistic defined on
the elements of $X$,
when can one find a smaller set $Y$ and two statistics
$s$ and $t$ such that
$$     X(q)   =   \sum_{w \in Y} q^{s(w)} \cdot (1+q)^{t(w)}  .   $$
Such an interpretation is called an {\em $q$-$(1+q)$-analogue}.
Examples of $q$-$(1+q)$-analogues have been determined for
the $q$-binomial by Fu, Reiner, 
Stanton and Thiem~\cite{Fu_Reiner_Stanton_Thiem},
and for the $q$-Stirling numbers of the first and second kinds
by Cai and Readdy~\cite{Cai_Readdy},
who also gave poset and homotopy interpretations of their
$q$-$(1+q)$-analogues.

In 1916
MacMahon~\cite{MacMahon_1,MacMahon_collected,MacMahon_book}
observed that the Gaussian coefficient, also known as
the $q$-binomial coefficient, is given by
$$ \Gaussian{n}{k}
=
  \sum_{w \in \Omega_{n,k}}
        q^{\inv(w)}   .  $$
Here $\Omega_{n,k} = \SSSS(0^{n-k},1^{k})$ denotes all
permutations of the multiset $\{0^{n-k},1^{k}\}$,
that is, all words $w = w_1 \cdots w_n$
of length~$n$ with $n-k$ zeroes and
$k$ ones,
and
$\inv( \cdot )$ denotes the inversion statistic
defined by
$\inv(w_{1} w_{2} \cdots w_{n})
= 
|\{(i,j) \: : \: 1 \leq i < j \leq n, w_{i} > w_{j}\}|$.
Fu et al.\ defined a subset
$\Omega^{\prime}_{n,k} \subseteq \Omega_{n,k}$
and two statistics $a$ and $b$ such that
$$ \Gaussian{n}{k}
=
  \sum_{w \in \Omega^{\prime}_{n,k}}
        q^{a(w)} \cdot (1+q)^{b(w)}   .  $$

In this paper we will return to the original study
by Fu et al.\ of the Gaussian coefficient.
We discover a more compact
$q$-$(1+q)$-analogue which, like the original
Gaussian coefficient, is also symmetric
in the variables $k$ and $n-k$.
See Corollary~\ref{corollary_symmetry}
and Theorem~\ref{theorem_compact}.
This symmetry was missing in Fu et al.'s original 
$q$-$(1+q)$-analogue.
We give a Boolean algebra decomposition of the
related poset $\Omega_{n,k}$.
Since this poset is a distributive lattice, 
in the last section we extend these
ideas 
to poset decompositions of any distributive lattice
and
other analogues.

\section{A poset interpretation}

In this section we consider the poset structure on
$0$-$1$-words in $\Omega_{n,k}$.
For further poset terminology and background,
we refer the reader to~\cite{EC1}.

We begin by making
the set of elements $\Omega_{n,k}$
into a graded poset
 by defining the
cover relation to be
$$  u \circ 01 \circ v \prec u \circ 10 \circ v  , $$
where $\circ$ denotes concatenation of words.
The word $0^{n-k} 1^{k}$ is the minimal element
and
the word $1^{k} 0^{n-k}$ is the maximal element
in the poset $\Omega_{n,k}$.
Furthermore, this poset is graded by
the inversion statistic.
This poset is simply the interval $[\hz, x]$
of Young's lattice, where the minimal element
$\hz$ is the empty Ferrers diagram
and
$x$ is the Ferrers diagram consisting of $n-k$ columns and
$k$ rows.

An alternative description of the 
poset $\Omega_{n,k}$
is that it is isomorphic to the Birkhoff transform of the Cartesian
product of two chains.
Let $C_{m}$ denote the $m$-element chain.
The poset $\Omega_{n,k}$ is isomorphic to
the distributive lattice of all lower order ideals of
the product $C_{n-k} \times C_{k}$, usually 
denoted by $J(C_{n-k} \times C_{k})$.

\begin{definition}
Let $\Omega^{\doubleprime}_{n,k}$
consist of all $0$,$1$-words
$v = v_{1} v_{2} \cdots v_{n}$
in
$\Omega_{n,k}$
such that
$$
v_{1} \leq v_{2}, \:\:
v_{3} \leq v_{4}, \:\:
\ldots, \:\:
v_{2 \cdot \lfloor n/2 \rfloor - 1} \leq v_{2 \cdot \lfloor n/2 \rfloor} .
$$
\end{definition}
Observe that when $n$ is odd there is no condition
on the last entry $w_{n}$.
Define two maps $\phi$ and~$\psi$
on $\Omega_{n,k}$
by sending the word
$w = w_{1} w_{2} \cdots w_{n}$
to
\begin{align*}
\phi(w)
&= 
\min(w_{1},w_{2}), \max(w_{1},w_{2}), \:\:
\min(w_{3},w_{4}), \max(w_{3},w_{4}), \:\: \ldots  , \\
\psi(w)
&= 
\max(w_{1},w_{2}), \min(w_{1},w_{2}), \:\:
\max(w_{3},w_{4}), \min(w_{3},w_{4}), \:\: \ldots  .
\end{align*}
The map $\phi$ sorts the entries in positions $1$
and $2$, $3$ and $4$, and so on. If $n$ is odd, the entry~$w_{n}$
remains in the same position.
Similarly, the map $\psi$ sorts in reverse order each pair of positions.
Note that the map $\phi$ maps 
$\Omega_{n,k}$ surjectively onto the set
$\Omega^{\doubleprime}_{n,k}$.

We have the following  Boolean algebra decomposition
of the poset $\Omega_{n,k}$.

\begin{theorem}
The distributive lattice
$\Omega_{n,k}$ has the Boolean algebra decomposition
$$ \Omega_{n,k}
=
  \bigcup_{v \in \Omega^{\doubleprime}_{n,k}}
        [v,\psi(v)] .
$$
\label{theorem_decomposition}
\end{theorem}
\begin{proof}
Observe that the maps
$\phi$ and $\psi$ 
satisfy the inequalities $\phi(w) \leq w \leq \psi(w)$.
Furthermore, the fiber of the map
$\phi : \Omega_{n,k} \longrightarrow \Omega^{\doubleprime}_{n,k}$
is isomorphic to a Boolean algebra,
that is,
$\phi^{-1}(v) \cong [v,\psi(v)]$.
\end{proof}

For $v \in \Omega^{\doubleprime}_{n,k}$
define the statistic
$$  \ascodd(v)
   =
      |\{ i \: : \: v_{i} < v_{i+1}, i \text{ odd}\}|  ,  $$
that is, $\ascodd(\cdot)$ enumerates
the number of ascents in odd positions.

\begin{corollary}
The $q$-binomial is given by
\begin{equation}
\label{equation_smaller_q_binomial}
 \Gaussian{n}{k}
=
  \sum_{v \in \Omega^{\doubleprime}_{n,k}}
        q^{\inv(v)} \cdot (1+q)^{\ascodd(v)}   .  
\end{equation}
\label{corollary_q_1_q}
\end{corollary}
\begin{proof}
It is enough to observe that
the sum of the inversion
statistic over the elements in the fiber $\phi^{-1}(v) = [v,\psi(v)]$
for $v \in \Omega^{\doubleprime}_{n,k}$
is given by
$q^{\inv(v)} \cdot (1+q)^{\ascodd(v)}$.
\end{proof}

A geometric way to understand this $q$-$(1+q)$-interpretation
is to consider lattice paths from the origin $(0,0)$
to $(n-k,k)$ which only use east steps $(1,0)$ and
north steps $(0,1)$. Color the squares of this $(n-k) \times k$
board as a chessboard, where the square incident to
the origin is colored white.
The map $\phi$ in the proof
of Theorem~\ref{theorem_decomposition}
corresponds to taking
a lattice path where every time
there is a north step followed by an
east step that 
turns around a white square, we exchange these two steps.
The statistic $\ascodd$ enumerates
the number of times an east step is followed by a north step
when this pair of steps 
borders a white square.

Let $\er(n,k)$ denote the cardinality of the set
$\Omega^{\doubleprime}_{n,k}$.
Then we have
\begin{proposition}
The cardinalities $\er(n,k)$ satisfy
the recursion
\begin{align*}
\er(n,k) & = \er(n-2,k-2) + \er(n-2,k-1) + \er(n-2,k) 
\:\:\:\:
\text{ for } 0 \leq k \leq n \text{ and } n \geq 2,
\end{align*}
with the boundary conditions $\er(0,0) = \er(1,0) = \er(1,1) = 1$
and
$\er(n,k) = 0$ whenever $k > n$, $k<0$ or $n < 0$.
\end{proposition}
\begin{proof}
A word in
$\Omega^{\doubleprime}_{n,k}$
begins with either $00$, $01$ or $11$,
yielding the three cases of the recursion.
\end{proof}

Directly we obtain the generating polynomial.
\begin{theorem}
The generating polynomial for $\er(n,k)$
is given by
$$   \sum_{k=0}^{n} \er(n,k) \cdot x^{k}
   =
     (1+x+x^{2})^{\lfloor n/2 \rfloor}
   \cdot
     (1+x)^{n - 2 \cdot \lfloor n/2 \rfloor}  .  $$
\end{theorem}

We end with a statement concerning the symmetry of
the $q$-$(1+q)$-binomial.

\begin{corollary}
The set of defining elements for the $q$-$(1+q)$-binomial
satisfy the following symmetric relation:
$$
     |\Omega_{n,k}''| = |\Omega_{n,n-k}''| .
$$
\label{corollary_symmetry}
\end{corollary}
\begin{proof}
This follows from the fact that
the generating polynomial for
$\er(n,k)$ is a product of palindromic polynomials,
and thus is itself is a palindromic polynomial. 
\end{proof}

\section{Analysis of the Fu--Reiner--Stanton--Thiem interpretation}

A {\em weak partition} is a finite non-decreasing sequence
of non-negative integers. 
A weak partition~$\lambda = (\lambda_{1}, \ldots, \lambda_{n-k})$
with $n-k$ parts and each part at most $k$
where $\lambda_1 \leq \cdots \leq \lambda_{n-k}$
corresponds to a Ferrers diagram lying inside
an
$(n-k) \times k$ rectangle
with column $i$ having height~$\lambda_i$.
These weak partitions are in direct correspondence
with the set $\Omega_{n,k}$.

Fu, Reiner, Stanton and Thiem used 
a pairing algorithm to determine
a subset $\Omega_{n,k}' \subseteq \Omega_{n,k}$
of $0$-$1$-sequences
to define their $q$-$(1+q)$-analogue of the $q$-binomial; 
see~\cite[Proposition~6.1]{Fu_Reiner_Stanton_Thiem}.
This translates into
the following statement.
The set $\Omega_{n,k}'$
is in bijection with 
weak partitions into $n-k$ parts
with each part at most $k$
such that
\begin{itemize}
\item[(a)] if $k$ is even, each odd part has even multiplicity,
\item[(b)] if $k$ is odd, each even part (including $0$)
has even multiplicity.
\end{itemize}

\begin{definition}
Let $\frst(n,k)$ be the cardinality of
the set $\Omega_{n,k}'$.
\end{definition}

\begin{lemma}
The quantity $\frst(n,k)$ counts the number of
weak partitions into $n-k$ parts
where each part is at most $k$
and each odd part has even multiplicity.
\label{lemma_complement}
\end{lemma}
\begin{proof}
When $k$ is even there is nothing to prove.
When $k$ is odd,
by considering the complement of weak partitions
with respect to the rectangle of size $(n-k) \times k$,
we obtain a bijective proof.
The same complement proof also shows the case when
$k$ is even holds.
\end{proof}

\begin{theorem}
The $\frst$-coefficients satisfy the recursion
\begin{align*}
\frst(n,k) & = \frst(n-1,k-1) + \frst(n-1,k)
&& \text{ for $k$ even,} \\
\frst(n,k) & = \frst(n-2,k-2) + \frst(n-2,k-1) + \frst(n-2,k)
&& \text{ for $k$ odd,}
\end{align*}
where $0 \leq k \leq n$ 
and $n \geq 2$
with the boundary conditions
$\frst(0,0) = \frst(1,0) = \frst(1,1) = 1$
and
$\frst(n,k) = 0$ whenever $k > n$, $k<0$ or $n < 0$.
\end{theorem}
\begin{proof}
We use the characterization 
in Lemma~\ref{lemma_complement}.
When $k$ is even there are two cases.
If the last part is $k$, remove it to obtain 
a weak partition counted by $\frst(n-1,k)$.
If the last part is less than~$k$, then the weak partition is counted
by $\frst(n-1,k-1)$.

When $k$ is odd there are three cases.
If the last two parts are equal to $k$, then removing
these two parts yields a weak partition counted by
$\frst(n-2,k)$.
Note that we cannot have the last part equal to $k$
and the next to last part less than $k$ since $k$ is odd.
If the last part is equal to $k-1$, we can remove it
to obtain a weak partition counted by
$\frst(n-2,k-1)$. Finally, if the last part is less than or equal
to $k-2$, the weak partition is counted by
$\frst(n-2,k-2)$.
\end{proof}

\begin{remark}
{\rm
For $k$ odd we have the shorter recursion
$\frst(n,k) = \frst(n-1,k-1) + \frst(n-2,k)$.
However, we use the longer recursion in
the proof of Theorem~\ref{theorem_compact}.
}
\end{remark}

\begin{lemma}
The inequality
$\frst(n,k) \leq \frst(n+1,k+1)$
holds.
\label{lemma_frst_inequality}
\end{lemma}
\begin{proof}
The weak partitions which lie inside the rectangle $(n-k) \times k$
and satisfy  the conditions of Lemma~\ref{lemma_complement}
are included among the weak partitions which lie inside the
larger 
rectangle $(n-k) \times (k+1)$ and satisfy the same conditions.
\end{proof}

\begin{theorem}
For all $0 \leq k \leq n$ the inequality
$|\Omega_{n,k}''| = \er(n,k) \leq \frst(n,k) = |\Omega_{n,k}'|$
holds.
\label{theorem_compact}
\end{theorem}
\begin{proof}
We proceed by induction on $n$. The induction base is $n \leq 3$.
Furthermore, the inequality holds when $k$ is $0$, $1$, $n-1$ and $n$.
When $k$ is odd we have that
\begin{align*}
\er(n,k) & = \er(n-2,k-2) + \er(n-2,k-1) + \er(n-2,k) \\
         & \leq \frst(n-2,k-2) + \frst(n-2,k-1) + \frst(n-2,k) \\
         & = \frst(n,k) .
\end{align*}
Similarly, when $k$ is even we have
\begin{align*}
\er(n,k) & = \er(n-2,k-2) + \er(n-2,k-1) + \er(n-2,k) \\
         & \leq \frst(n-2,k-2) + \frst(n-2,k-1) + \frst(n-2,k) \\
         & \leq \frst(n-1,k-1) + \frst(n-2,k-1) + \frst(n-2,k) \\
         & = \frst(n-1,k-1) + \frst(n-1,k) \\
         & = \frst(n,k) ,
\end{align*}
where the second inequality follows from
Lemma~\ref{lemma_frst_inequality}.
These two cases complete the induction hypothesis.
\end{proof}

See Table~\ref{table_frst_er} to compare the values
of $\frst(n,k)$ and $\er(n,k)$ for $n \leq 10$.

\begin{table}
$$
\begin{array}{r r r r r r r r r r r c r r r r r r r r r r r}
1 &&&&&&&&&&& \hspace*{10 mm} &
1 \\
1 & 1 &&&&&&&&&&&
1 & 1 \\
1 & 1 & 1 &&&&&&&&&&
1 & 1 & 1 \\
1 & 2 & 2 & 1 &&&&&&&&&
1 & 2 & 2 & 1 \\
1 & 2 & 4 & 2 & 1 &&&&&&&&
1 & 2 & 3 & 2 & 1 \\
1 & 3 & 6 & 5 & 3 & 1 &&&&&&&
1 & 3 & 5 & 5 & 3 & 1 \\
1 & 3 & 9 & 8 & 8 & 3 & 1 &&&&&&
1 & 3 & 6 & 7 & 6 & 3 & 1 \\
1 & 4 & 12 & 14 & 16 & 9 & 4 & 1 &&&&&
1 & 4 & 9 & 13 & 13 & 9 & 4 & 1 \\
1 & 4 & 16 & 20 & 30 & 19 & 13 & 4 & 1 &&&&
1 & 4 & 10 & 16 & 19 & 16 & 10 & 4 & 1 \\
1 & 5 & 20 & 30 & 50 & 39 & 32 & 14 & 5 & 1 &&&
1 & 5 & 14 & 26 & 35 & 35 & 26 & 14 & 5 & 1 \\
1 & 5 & 25 & 40 & 80 & 69 & 71 & 36 & 19 & 5 & 1 &&
1 & 5 & 15 & 30 & 45 & 51 & 45 & 30 & 15 & 5 & 1 \end{array}
$$
\caption{The $\frst$- and $\er$-triangles for $n \leq 10$.}
\label{table_frst_er}
\end{table}

\section{Concluding remarks}

Is it possible to find 
a $q$-$(1+q)$-analogue of the Gaussian coefficient
which has the smallest possible index set?
We believe that our analogue is the smallest,
but cannot offer a proof of a minimality.
Perhaps a more tractable question is to prove
that the Boolean algebra
decomposition of
$\Omega_{n,k}$ is minimal.

We can extend these ideas
involving a Boolean algebra decomposition
to any distributive lattice. Let~$P$ be a finite poset
and
let $A$ be an antichain of $P$ such that
there is no cover relation in $A$, that is,
there is no pair of elements $u, v \in A$ such that $u \prec v$.
We obtain a Boolean algebra decomposition of the
Birkhoff transform $J(P)$ by defining
$$  J^{\doubleprime}(P)
    =
      \{I \in J(P) \:\: : \:\:
          \text{ the ideal $I$ has no maximal elements
                       in the antichain $A$}\} .  $$
The two maps $\phi$ and $\psi$ are now defined as
\begin{align*}
\phi(I)
& =
I - \{a \in A \: : \: \text{ the element $a$ is maximal in $I$}\} , \\
\psi(I)
& =
I \cup \{a \in A \: : \: I \cup \{a\} \in J(P)\} .
\end{align*}
We have the following decomposition theorem.
\begin{theorem}
For $P$ any
finite poset the distributive lattice $J(P)$ has the Boolean algebra
decomposition
$$   J(P)
  =
\bigcup_{I \in J^{\doubleprime}(P)}
      [I, \psi(I)]   .  $$
\end{theorem}

Yet again, how can we select the antichain $A$
such that the above decomposition $A$
has the fewest possible terms?
Furthermore, would this give the smallest
Boolean algebra decomposition?

Another way to extend the ideas of
Theorem~\ref{theorem_decomposition}
is as follows.
Define $\Omega^{r}_{n,k}$
to be the set of all words $v \in \Omega_{n,k}$
satisfying the inequalities
$$
v_{1} \leq v_{2} \leq \cdots \leq v_{r}, \:\:
v_{r+1} \leq v_{r+2} \leq \cdots \leq v_{2r}, \:\:
\ldots, \:\:
v_{r \cdot \lfloor n/r \rfloor - r+1}
\leq v_{r \cdot \lfloor n/r \rfloor - r+2}
\leq \cdots 
\leq v_{r \cdot \lfloor n/r \rfloor} .
$$
For $1 \leq i \leq \lfloor r/2\rfloor$
define the statistics
$b_{i}(v)$
for $v \in \Omega^{r}_{n,k}$
to be
$$
b_{i}(v)
  =
|\{j \in [\lfloor n/r \rfloor]
\: : \:
v_{r j - r+1} + v_{r  j - r+2} + \cdots + v_{r j} \in \{i,r-i\}\}| .
$$
\begin{theorem}
The distributive lattice $\Omega_{n,k}$
has the decomposition
$$ \Omega_{n,k}
   =
  \bigcup_{v \in \Omega^{r}_{n,k}}
  \Omega_{r,1}^{b_{1}(v)}
\times
  \Omega_{r,2}^{b_{2}(v)}
\timesdots
  \Omega_{r,\lfloor r/2 \rfloor}^{b_{\lfloor r/2 \rfloor}(v)}  .  $$
\label{theorem_r}
\end{theorem}
\begin{corollary}
The $q$-binomial is given
by
$$ \Gaussian{n}{k}
=
  \sum_{v \in \Omega^{r}_{n,k}}
        q^{\inv(v)}
\cdot
  \Gaussian{r}{1}^{b_{1}(v)}
\cdot
  \Gaussian{r}{2}^{b_{2}(v)}
\cdots
  \Gaussian{r}{\lfloor r/2 \rfloor}^{b_{\lfloor r/2 \rfloor}(v)}  .  $$
\label{corollary_r}
\end{corollary}
The least complicated case is when $r=3$,
where only one term appears in the above poset product.
This term is $\Omega_{3,1}$ which is the three element chain $C_{3}$.
The associated Gaussian coefficient
is $1+q+q^{2}$. 
Thus Corollary~\ref{corollary_r}
could be called a {\em $q$-$(1+q+q^{2})$-analogue}
in the case of $r = 3$.
As an example, we have
$$ \Gaussian{6}{3}
  =  1 + q \cdot (1+q+q^{2})^{2} + q^{4} \cdot (1+q+q^{2})^{2} + q^{9} . $$
On a poset level this is a decomposition
of $J(C_{3} \times C_{3})$ into two
one-element posets of rank~$0$ and rank~$9$,
and two copies of $C_{3} \times C_{3}$,
where one has its minimal element of rank~$1$ and the other
of rank~$4$.

\section*{Acknowledgements}

The authors thank the referee for helpful comments.
This work was partially supported by a grant from
the Simons Foundation (\#206001 to Margaret~Readdy).

\newcommand{\journal}[6]{{\sc #1,} #2, {\it #3} {\bf #4} (#5), #6.}
\newcommand{\preprint}[3]{{\sc #1,} #2, preprint #3.}
\newcommand{\book}[4]{{\sc #1,} #2, #3, #4.}
\newcommand{\collection}[6]{{\sc #1,}  #2, #3, in {\it #4}, #5, #6.}
\newcommand{\JCTA}{J.\ Combin.\ Theory Ser.\ A}
\newcommand{\arxiv}[3]{{\sc #1,} #2, {\tt #3}.}

\end{document}